\documentclass[11pt]{article}

\title{Finite subgroups of automorphisms of free products}
\author{Ioannis Papavasileiou \and Dionysios Syrigos}

\usepackage[top=2cm, bottom=4.5cm, left=2.5cm, right=2.5cm]{geometry}

\usepackage{graphicx}
\usepackage{setspace}

\usepackage{filecontents}
\usepackage{ccfonts}
\usepackage[T1]{fontenc}

\usepackage{mathtools}
\usepackage{lipsum}
\usepackage{xfrac}
\usepackage{amsfonts}
\usepackage{cancel}
\usepackage{soul}
\usepackage{graphics}
\usepackage{multicol}
\usepackage{verbatim}
\usepackage{float}
\usepackage{amsthm}
\usepackage{amsmath}
\usepackage{hyperref}
\usepackage{amssymb}
\usepackage[dvipsnames]{xcolor}
\usepackage{color}

\usepackage[
backend=biber,
style=numeric,
sorting=nty
]{biblatex}
\usepackage[draft]{todonotes}

\usepackage{soul}

\DeclareFieldFormat{pages}{#1}
\DeclareFieldFormat[inbook]{citebookTitle}{\mkbibquote{#1\isdot}}
\DeclareFieldFormat[inbook]{bookTitle}{\mkbibquote{#1\isdot}} 
\DeclareFieldFormat[inbook]{citebookTitle}{#1}
\DeclareFieldFormat[inbook]{bookTitle}{#1} 

\renewbibmacro*{volume+number+eid}{
  \printfield{volume}
  \setunit*{\addnbspace} 
  \printfield{number}
  \setunit{\addcomma\space}
  \printfield{eid}}
\DeclareFieldFormat[article]{number}{\mkbibparens{#1}}

\renewbibmacro{in:}{
  \ifentrytype{article}
    {}
    {\bibstring{in}
     \printunit{\intitlepunct}}}

\hypersetup{
    colorlinks=true,
    linkcolor=blue,
    filecolor=blue,      
    urlcolor=blue,
    pdftitle={Finite_Subgroups_of_Out(G)},
    pdfpagemode=FullScreen,
    }

\newtheorem{lemma}{Lemma}[section]
\newtheorem{proposition}{Proposition}[section]
\newtheorem{corollary}{Corollary}[section]
\theoremstyle{definition}
\newtheorem{remark}{Remark}[section] 
\newtheorem{definition}{Definition}[section]
 
\newtheorem*{sublemma*}{Lemma}

\newenvironment{keywords}{}{}

\usepackage{titlesec}
\usepackage{titlecaps}
\usepackage{fancyhdr}

\newtheorem{theorem}{Theorem}[section]
\newtheorem{innercustomthm}{Theorem}
\newenvironment{mythm}[1]
  {\renewcommand\theinnercustomthm{#1}\innercustomthm}
  {\endinnercustomthm}

\makeatletter
\def\keywords{\xdef\@thefnmark{}\@footnotetext}
\makeatother

\begin{document}
\newpage
\maketitle
 
\begin{abstract}
We study finite subgroups of outer automorphisms of free products. We give upper bounds for the orders of these finite subgroups as well as bounds for the orders of individual torsion outer automorphisms under some (necessary) conditions for the free factors.
\end{abstract}
\setstretch{1.4}


\section{Introduction}
One of the major problems in group theory is to determine the structure of the group of automorphisms $\mathrm{Aut}(G)$ of a group $G$, especially its finite
subgroups  as well as the behavior of individual
automorphisms. For example, using Nielsen's Realization Theorem one can study finite subgroups of $\mathrm{Aut}(F_p)$ and $\mathrm{Out}(F_p)$. The
maximal order of a finite subgroup of  $\mathrm{Out}(F_p)$ is $2^p\cdot p!$ \cite{WZ} and is realized as
the group of automorphisms of the standard rose $R_p$. It can also been proved \cite{BAO2000437},\cite{LEVITT1998630} that the
maximal order of torsion elements in $\mathrm{Aut}(F_p)$  grows asymptotically like $\mathrm{exp}(\sqrt{p\log{p}})$.\par
In this note we study finite subgroups of the group of outer automorphisms $\mathrm{Out}(G)$ of a group $G$ with a free product decomposition $G=G_1\ast\ldots\ast G_k\ast F_p$. We denote by $\mathrm{Out}(G,\mathcal{G})$ the subgroup of $\mathrm{Out}(G)$ which consists of those outer automorphisms that permute the counjugacy classes of the free factors $G_1,\ldots,G_k$. We firstly give a (uniform) upper bound for the orders of its finite subgroups. 

\begin{mythm}{\ref{main}}
 Let $G$ be a finitely generated group and $G=G_1\ast\ldots\ast G_k\ast F_p$ a free product decomposition of $G$ such that for the $G_i$'s and $\mathrm{Aut}(G_i)$'s there is a (uniform) bound for the orders of their finite subgroups. Then, there is a (uniform) bound for the orders of finite subgroups of $\mathrm{Out}(G,\mathcal{G})$, too.
  \end{mythm}
The converse of the above theorem is also true (see Proposition \ref{invmain}), thus the assumptions made on the $G_i$'s and $\mathrm{Aut}(G_i)$'s are minimal.\par
Furthermore, we prove that finite subgroups of $\mathrm{Out}(G)$ have uniformly bounded orders in case where $G$ is either a finite extension of a finitely generated free group (see Corollary \ref{fbf}) or a finite extension of a free product as in Theorem \ref{main} (see Corollary \ref{vfp}).\par
Finally, we give an upper bound for the orders of individual outer automorphisms.

\begin{mythm}{\ref{second}}
 Let $G$ be a finitely generated group and $G=G_1\ast\ldots\ast G_k\ast F_p$ a free product decomposition of $G$. Let $M_i$ be the maximal order among the elements of finite order in $\mathrm{Out}(G_i)$ and 
 let $N_i$ be the corresponding maximal order among the elements of finite order in $G_i$. Then every torsion element of $\mathrm{Out}(G,\mathcal{G})$ has order at most $M=M(M_i,N_i,k,p)$, where 
 \[
 M=\bigg(\left[e^{\frac{k}{e}}\right]+1\bigg)\cdot  \prod_iM_i\cdot\mathrm{exp}\big((1+\theta_p)\sqrt{p\log{p}}\big)\cdot \big(\max\limits_i\left\{N_i\right\}\big)^{2k-2+2p}
 \]
  and $\lim\limits_{p\to+\infty}\theta_p=0$.
  \end{mythm}

\section{Preliminaries}\label{prelim}
\subsection{Free factor systems and automorphisms}
Let $G$ be a finitely generated group and $G=G_1\ast\ldots\ast G_k\ast F_p$ a free product decomposition of $G$. If this is the Grushko decompositions of $G$, then every automorphism of $G$ permutes the conjugacy classes of the $G_i$'s, i.e. if $\phi\in\mathrm{Aut}(G)$, then $\phi(G_i)=gG_jg^{-1}$ for some $g\in G$. Apart from the Grushko decomposition, there are other decompositions which could be useful, too. We let $\mathcal{G}=\{[G_1],\ldots,[G_k]\}$ be the set of $G$-congugacy classes of the subgroups $G_i$; we call it a \textbf{free factor system} of $G$. There is a natural action of the group of outer automorphisms $\mathrm{Out}(G)$ of $G$ on the set of free factor systems of $G$. If $[\phi]\in \mathrm{Out}(G)$ is an outer automorphism and $\mathcal{H}=\{[H_1],\ldots,[H_n]\}$ is a free factor system, then $[\phi](\mathcal{H})=\{[\phi(H_1)],\ldots,[\phi(H_n)]\}$ is a free factor system of $G$, too. In general, it is not true that $[\phi](\mathcal{H})=\mathcal{H}$. The \textbf{relative (to $\mathcal{G}$) outer automorphism group} $\mathrm{Out}(G,\mathcal{G})$ is the subgroup of $\mathrm{Out}(G)$ consisting of those outer automorphisms $[\phi]$ which preserve $\mathcal{G}$, i.e. $[\phi](\mathcal{G})=\mathcal{G}$.\par
Let $\mathrm{Out}'(G,\mathcal{G})$ be the finite index subgroup of $\mathrm{Out}(G,\mathcal{G})$ consisting of outer automorphisms fixing each congugacy class $[G_i]$. 
Since each $G_i$ equals to its
 normalizer in $G$, $\mathrm{Out}'(G,\mathcal{G})$ maps onto $\mathrm{Out}(G_i)$, for every $i$. There is also a homomorphism $\mathrm{Out}'(G,\mathcal{G})\to \mathrm{Out}(F_p)$, hence we have a natural map:
 \[
\pi:\mathrm{Out}'(G,\mathcal{G})\to \mathrm{Out}(F_p)\times \prod_{i}\mathrm{Out}(G_i).
\]
We will denote by $\mathrm{Out}(G,\mathcal{G}^{(t)})$ the kernel of the epimorphism $\mathrm{Out}'(G,\mathcal{G})\to\prod\limits_{i}\mathrm{Out}(G_i)$ and by $\mathrm{Out}_0(G,\mathcal{G}^{(t)})$ the kernel of the epimorphism $\pi$.

\subsection{The relative outer space}

We shall describe a space, which was introduced by Guirardel and Levitt in \cite{Outspaceprod}, that depends only on the chosen free factor system $\mathcal{G}$ of $G$.\par
We need the following definitions. Let $G$ be a group acting on a metric simplicial tree $T$, i.e. $T$ is a simplicial tree and each orbit of edges  is assigned with a positive length.
\begin{definition}
    We say that a metric simplicial tree $T$ is a \textbf{$G$-tree} if there exists a $G$-action on $T$ which is:
    \begin{enumerate}
        \item isometric, i.e. $d_T(x,y)=d_T(gx,gy)$ for every $x,y\in T$ and $g\in G$,
        \item simplicial, i.e. vertices are sent to vertices and edges to edges,
        \item minimal, i.e. there is no proper, non trivial $G$-invariant subtree of $T$ and
        \item co-compact, i.e. the quotient space $T/G$ is a finite graph.
     \end{enumerate}
\end{definition}

\begin{definition}
    Let $\mathcal{G}=\{[G_1],\ldots,[G_k]\}$ be a free factor system of $G$. A $G$-tree $T$ will be said to be a \textbf{$\mathcal{G}$-tree} if:
    \begin{enumerate}
        \item all edge stabilisers are trivial and
        \item the free factor system induced by the vertex stabilisers of $T$ is exactly $\mathcal{G}$, i.e. for each $i$ there is exactly one orbit of vertices with stabilizer conjugate to $G_i$ and all other points have trivial stabilizer.
    \end{enumerate}
\end{definition}
We are now in position to give the definition of the outer space relative to $\mathcal{G}$.
\begin{definition}
    The \textbf{(relative) outer space} $\mathcal{O}=\mathcal{O}(\mathcal{G})$ is defined to be the space of equivalence classes of $\mathcal{G}$-trees where two trees are considered as equivalent if there exists a $G$-equivariant homothety between them.
\end{definition}
In order to refer to a point of $\mathcal{O}(\mathcal{G})$ we shall write $T\in \mathcal{O}(\mathcal{G})$ instead of its class $[T]\in \mathcal{O}(\mathcal{G})$.\par
The group $\mathrm{Aut}(G,\mathcal{G})$, consisting of automorphisms which permute the conjugacy classes of the $G_i$'s,  acts on the set of $\mathcal{G}$-trees by changing the action, i.e. for $\phi\in \mathrm{Aut}(G,\mathcal{G})$ and a $\mathcal{G}$-tree $T$, $\phi(T)$ is defined to be the $\mathcal{G}$-tree with the same
underlying tree as $T$ and $G$-action given by $(g,x)\in G\times T\mapsto \phi(g)x\in T$. It is clear that inner automorphisms of $G$ act by isometries on the set of $\mathcal{G}$-trees hence there is an induced action of $\mathrm{Out}(G,\mathcal{G})$ on $\mathcal{O}(\mathcal{G})$. Roughly speaking, every element of $\mathcal{O}(\mathcal{G})$ can be viewed as a graph of groups and the action of $\mathrm{Out}(G,\mathcal{G})$ is by changing the marking.\par
Every $T\in \mathcal{O}$ determines an (open) simplex $\Delta(T)\subset\mathcal{O}$ which is the set of points of $\mathcal{O}$ obtained from $T$ by changing the lengths of (orbits of) edges in such
a way so that every edge has positive length and the sum of lengths of edges in the quotient graph $T/G$ is equal to $1$. In this way we can embed every open simplex $\Delta(T)$ of $\mathcal{O}$ in a Euclidean space $\mathbb{R}^n$, where $n$ depends on the number of orbits of edges of $T$. Thus, $\mathcal{O}$ can be thought of as a union of open simplices as described above. It is easy to see that  there are only finitely many orbits of simplices under the action of $\mathrm{Out}(G,\mathcal{G})$ on $\mathcal{O}(\mathcal{G})$.\par
We recall the description of the stabilizer of $\Delta(T)$, i.e. the subgroup $\mathrm{Out}^{\Delta(T)}(G,\mathcal{G})\leq \mathrm{Out}(G,\mathcal{G})$ sending $\Delta(T)$ to itself, as given in \cite{Outspaceprod}. Let $\mathrm{Out}_0^{\Delta(T)}(G,\mathcal{G})\leq \mathrm{Out}^{\Delta(T)}(G,\mathcal{G})$ be the finite index subgroup consisting of outer automorphisms acting trivially on the quotient graph $T/G$. We denote by $n_i$ the degree of every non-free vertex $v_i$ of $T/G$.
\begin{remark}\label{valences}
    It is well known that for every point $T\in \mathcal{O}$, the sum $\sum\limits_{v_i\in V(T/G)}n_i$ is at most $2(3k-3+2p)$ \cite{SykStableRepresentatives,Lyman}. In fact, it is not difficult to see that the above bound can be improved to $2k-2+2p$, since the minimum number of orbits of edges of a point in $\mathcal{O}$ is $k-1+p$.
\end{remark}
The group $\mathrm{Out}_0^{\Delta(T)}(G,\mathcal{G})$ is a direct product $\prod\limits_{i=1}^kM_{n_i}(G_i)$ of groups which fit in exact sequences
 \begin{align}
     \{1\}\to G_i^{n_i}/Z(G_i)\to M_{n_i}(G_i)\to \mathrm{Out}(G_i)\to  \{1\}\label{(1)}\\
     \{1\}\to G_i^{n_i-1}\to M_{n_i}(G_i)\to \mathrm{Aut}(G_i)\to \{1\}\label{(2)}
 \end{align}
 where the center $Z(G_i)$ is embedded diagonally into $G_i^{n_i}$. The latter exact sequence is split hence $M_{n_i}(G_i)=G_i^{n_i-1}\rtimes \mathrm{Aut}(G_i)$ where $\mathrm{Aut}(G_i)$ acts diagonally on $G_i^{n_i-1}$. It can be proved \cite[Proposition~4.2]{Levitthyper} that 
 \begin{align}\label{(3)}
  \mathrm{Out}_0^{\Delta(T)}(G,\mathcal{G}) =\prod_{i=1}^kM_{n_i}(G_i)=\prod_{i=1}^k\big(G_i^{n_i-1}\rtimes \mathrm{Aut}(G_i)\big).
\end{align}
In particular, the kernel of the product epimorphism 
\begin{align}\label{(4)}
p:\mathrm{Out}_0^{\Delta(T)}(G,\mathcal{G}) =\prod_{i=1}^kM_{n_i}(G_i) \to \prod_{i=1}^k \mathrm{Aut}(G_i)
\end{align}
is $\mathrm{ker}p=\prod\limits_{i=1}^kG_i^{n_i-1}$. \par
We also have the following embedding (see the proof of Lemma 5.1 in \cite{Outspaceprod})
 \begin{align}\label{(5)}
\mathrm{Out}_0^{\Delta(T)}(G,\mathcal{G}^{(t)})\coloneqq \mathrm{Out}_0(G,\mathcal{G}^{(t)})\cap \mathrm{Out}_0^{\Delta(T)}(G,\mathcal{G})\hookrightarrow \prod_{i=1}^k(G_i^{n_i}/Z(G_i)).
\end{align}

The following result, which is the Nielsen Realisation for free
products, is a key ingredient for the proof of our main theorem.
 \begin{theorem}[{{\cite[Corollary 6.1.]{Hensel_2018}}}]\label{finitesub}
     Let $G$ be a finitely generated group and $G=G_1\ast\ldots\ast G_j\ast F_p$ a free product decomposition of $G$. Every finite subgroup $H\leq \mathrm{Out}(G,\mathcal{G})$ fixes a point of $\mathcal{O}(\mathcal{G})$.
 \end{theorem}

\subsection{Finite subgroups}
Let $G$ be a group and $g$ a torsion element of $G$. We denote by $o(g)$ the order of $g$ in $G$ and by $T(G)$ the set of all torsion elements of $G$.
\begin{lemma}\label{lcm}
    Let $G$ be a direct product $G=\prod_{1\leq i\leq n} G_i$ where the $G_i$'s have an upper bound $M_i$ for the orders of their torsion elements. Then, torsion elements of $G$ have order at most
    \[
    \max\limits_{(g_i)_{1\leq i\leq n}} \mathrm{lcm}\left\{o(g_1),\ldots,o(g_n)\right\}\leq \prod_iM_i
    \]
\end{lemma}
\begin{proof}
    Note that if $g=(g_1,\ldots,g_n)$ is a torsion element of $G$, then $o(g)=\mathrm{lcm}\left\{o(g_1),\ldots,o(g_n)\right\}$.
\end{proof}

\begin{remark}
There are cases where the above inequality is achieved as equality. For example, take $\mathbb{Z}_2\times\mathbb{Z}_3$. On the other hand, if two of the $G_i$'s are isomorphic, then the inequality is proper. In fact in a direct product of the form $\prod_{1\leq i\leq n} G_i=A^n$ the torsion elements have order at most
    \[
    M(M-1)\cdot\ldots\cdot (M-n+1)=(M)_n
    \]
    where $M$ is an upper bound for the order of the torsion elements of $A$.
\end{remark}

\begin{lemma}\label{direct}
    Let $G$ be a direct product $G=\prod_{1\leq i\leq n} G_i$ where the $G_i$'s have an upper bound $M_i$ for the orders of their finite subgroups. Then, every finite subgroup of $G$ has order at most $\prod\limits_{i=1}^nM_i$.
\end{lemma}
\begin{proof}
     We denote by $p_i:G\to G_i$ the natural projections. Let $H$ be a finite subgroup of $G$. Consider the homomorphism
    \[
    H \to p_1(H)\times\cdots\times p_n(H), \text{ given by }  h\mapsto \Big(p_1(h),\ldots,p_n(h)\Big). 
    \]
    This map is injective, and each image $p_i(H)\cong H/\mathrm{ker}p_i|_H$ is a finite subgroup of $G_i$. Hence,
    \[
    |H|\leq \prod_{i=1}^n|p_i(H)|\leq \prod_{i=1}^nM_i.
    \]
\end{proof}

 \section{The main theorem}
 We now give an upper bound on the maximal order of finite subgroups of $\mathrm{Out}(G,\mathcal{G})$.
 \begin{theorem}\label{main}
     Let $G$ be a finitely generated group and $G=G_1\ast\ldots\ast G_k\ast F_p$ a free product decomposition of $G$ such that for the $G_i$'s and $\mathrm{Aut}(G_i)$'s there is a (uniform) bound for the orders of their finite subgroups. Then, there is a (uniform) bound for the orders of finite subgroups of $\mathrm{Out}(G,\mathcal{G})$, too.
 \end{theorem}
 \begin{proof}
Let $H$ be a finite subgroup of $\mathrm{Out}(G,\mathcal{G})$. By Theorem \ref{finitesub}, $H$ fixes a point $\Delta(T)$ of $\mathcal{O}(\mathcal{G})$, i.e. $H\leq \mathrm{Out}^{\Delta(T)}(G,\mathcal{G})$. We will give an upper bound for the cardinality of $H$ in the following four steps.
\begin{enumerate}
    \item In order to make use of the strong properties of $\mathrm{Out}_0^{\Delta(T)}(G,\mathcal{G})$ which is of finite index in $\mathrm{Out}^{\Delta(T)}(G,\mathcal{G})$, we first consider the intersection $H_1=H\cap \mathrm{Out}'(G,\mathcal{G})$. Then $H/H_1$ is isomorphic to a subgroup of $\prod_{i=1}^m \mathrm{Sym}_{k_i}$ (where $m$ is the number of isomorphic classes of the $G_i$'s and $k_i$ is the number of appearances of the $i$-th class) since 
    \[
    H/H_1=\dfrac{H}{H\cap \mathrm{Out}'(G,\mathcal{G})} \cong \dfrac{H\cdot\mathrm{Out}'(G,\mathcal{G})}{\mathrm{Out}'(G,\mathcal{G})}\leq \dfrac{\mathrm{Out}(G,\mathcal{G})}{\mathrm{Out}'(G,\mathcal{G})}\leq \prod_{i=1}^m \mathrm{Sym}_{k_i}.
    \]
    We denote by $A$ the order of a finite subgroup of maximal order in $\prod_{i=1}^m \mathrm{Sym}_{k_i}$. Note that $A\leq \prod\limits_{i=1}^mk_i!$.
    \item Let $H_0$ be the intersection $H_1\cap \mathrm{Out}_0^{\Delta(T)}(G,\mathcal{G})$. Then, $H_0$ acts trivially on the quotient graph $T/G$ and has finite index in $H_1$, since
    \[
    |H_1:H_0|=|H_1\cdot \mathrm{Out}_0^{\Delta(T)}(G,\mathcal{G}):\mathrm{Out}_0^{\Delta(T)}(G,\mathcal{G})|\leq |\mathrm{Out}^{\Delta(T)}(G,\mathcal{G}):\mathrm{Out}_0^{\Delta(T)}(G,\mathcal{G})|<+\infty.
    \]
    In fact, since $H_1$ fixes the vertices of $T/G$, the index $|H_1:H_0|$ depends only on $p=\mathrm{rank}(F_p)$.
    \item Since $H_0$ is a subgroup of $\mathrm{Out}_0^{\Delta(T)}(G,\mathcal{G})$ we can consider the restriction of the homomorphism $p$ to $H_0$, namely  $p|_{H_0}:H_0\to\prod\limits_{i=1}^k\mathrm{Aut}(G_i)$. By the first isomorphism theorem, $H_0/\mathrm{ker}p|_{H_0}\cong \mathrm{Im}p|_{H_0}\leq  \prod\limits_{i=1}^k\mathrm{Aut}(G_i)$. Since $H_0$ is finite (being a subgroup of $H$), the image $\mathrm{Im}p|_{H_0}$ is a finite subgroup of $\prod\limits_{i=1}^k\mathrm{Aut}(G_i)$, and therefore, by Lemma \ref{direct}, its cardinality is bounded by the product $\prod\limits_{i=1}^kM_i$ where each $M_i$ is an upper bound for the orders of finite subgroups of $\mathrm{Aut}(G_i)$.
    \item Finally we bound the order of the finite group $\mathrm{ker}p|_{H_0}$. By (\ref{(4)}) $\mathrm{ker}p|_{H_0}=\mathrm{ker}p\cap H_0 \leq \prod\limits_{i=1}^kG_i^{n_i-1}$, and thus, once again by Lemma \ref{direct}, the order $|\mathrm{ker}p|_{H_0}|$ is bounded by the product $\prod\limits_{i=1}^kN_i^{n_i-1}$, where each $N_i$ is an upper bound for the orders of finite subgroups of $G_i$. By Remark \ref{valences}, the product $\prod\limits_{i=1}^kN_i^{n_i-1}$ is bounded above by $\Big(\max\limits_i\{N_i\}\Big)^{k-2+2p}$.
\end{enumerate}
Since 
    \[
|H|=|H/H_1|\cdot |H_1/H_0|\cdot |H_0/\mathrm{ker}p|_{H_0}|\cdot |\mathrm{ker}|_{H_0}|
\]
the above steps complete the proof.
 \end{proof}
 \begin{remark}
     Note that the uniform bound $M$ for the orders of the finite subgroups of $\mathrm{Out}(G,\mathcal{G})$ given in the above proof depends only on the $G_i$'s, their automorphism groups, and $p=\mathrm{rank}(F_p)$.
 \end{remark}

We also note that the existence of a uniform bound on the orders of finite subgroups of $\mathrm{Out}(G,\mathcal{G})$ can also be proved by imposing certain hypothesis on the virtual cohomological dimension of the $G_i$'s and their automorphism groups.
\begin{remark}
 Let $G$ be a finitely generated group and $G=G_1\ast\ldots\ast G_k\ast F_p$ a free product decomposition of $G$. If 
    \begin{itemize}
        \item either, the $G_i$'s, $G_i/Z(G_i)$'s are  torsion free with finite virtual cohomological dimension and the $\mathrm{Out}(G_i)$'s are virtually
torsion-free  with finite virtual cohomological dimension,
\item or, the $G_i$'s, $G_i/Z(G_i)$'s are  torsion free and the $G_i$'s, $\mathrm{Aut}(G_i)$'s have finite virtual cohomological dimension,
    \end{itemize}
    then (see \autocite[Theorem~5.2]{Outspaceprod}) $\mathrm{Out}(G,\mathcal{G})$ has finite virtual cohomological dimension. In this case, there exists a finite index, torsion free, normal subgroup $H$ of $\mathrm{Out}(G,\mathcal{G})$. Thus, every finite subgroup $K$ of $\mathrm{Out}(G,\mathcal{G})$ embeds into $\mathrm{Out}(G,\mathcal{G})/H$ hence it has order at most $|\mathrm{Out}(G,\mathcal{G}):H|$. 
\end{remark}
 On the other hand,  Theorem \ref{main} makes the minimal (necessary and sufficient) assumptions required on the $G_i$'s and $\mathrm{Aut}(G_i)$'s as shown by the following proposition.
\begin{proposition}\label{invmain}
 Let $G$ be a finitely generated group and $G=G_1\ast\ldots\ast G_k\ast F_p$ a free product decomposition of $G$. There is a uniform bound for the orders of finite subgroups of $\mathrm{Out}(G,\mathcal{G})$ if and only if the same holds for the finite subgroups of the $G_i$'s and $\mathrm{Aut}(G_i)$'s.
\end{proposition}
\begin{proof}
    We shall prove that there is a uniform upper bound for the orders of the finite subgroups of the $G_i$'s and $\mathrm{Aut}(G_i)$'s, given that the same is true for the finite subgroups of $\mathrm{Out}(G,\mathcal{G})$. 
    \begin{itemize}
        \item If $p=2$ and $k=0$, then $\mathrm{Out}(G,\mathcal{G})=\mathrm{Out}(F_2)$ hence there are no free factors $G_i$'s and we have nothing to prove.
        \item If $(k,p)=(1,1)$ or $(2,0)$, then $\mathcal{O}(\mathcal{G})$ consists of only one simplex $\Delta(T)$. The groups $G_i$ and $\mathrm{Aut}(G_i)$ act on $\mathcal{O}(\mathcal{G})$ and stabilise $\Delta(T)$. Hence, $G_i$ and $\mathrm{Aut}(G_i)$ can be embedded into $\mathrm{Out}^{\Delta(T)}(G,\mathcal{G})$ and hence into $\mathrm{Out}(G,\mathcal{G})$.
        \item If $p+k\ge3$, then the $G_i$'s and $\mathrm{Aut}(G_i)$'s are isomorphic to subgroups of $\mathrm{Out}(G,\mathcal{G})$ (see \cite{Outspaceprod}), and hence there is an upper bound on the orders of the finite subgroups of the $G_i$'s and $\mathrm{Aut}(G_i)$'s.
    \end{itemize}
\end{proof}

The following theorem gives explicit upper bounds for the order of an outer automorphism of a free product.

\begin{theorem}\label{second}
    Let $G$ be a finitely generated group and $G=G_1\ast\ldots\ast G_k\ast F_p$ a free product decomposition of $G$. Let $M_i$ be the maximal order among the elements of finite order in $\mathrm{Out}(G_i)$ and 
 let $N_i$ be the corresponding maximal order among the elements of finite order in $G_i$. Then every torsion element of $\mathrm{Out}(G,\mathcal{G})$ has order at most $M=M(M_i,N_i,k,p)$, where 
 \[
 M=\bigg(\left[e^{\frac{k}{e}}\right]+1\bigg)\cdot  \prod_iM_i\cdot\mathrm{exp}\big((1+\theta_p)\sqrt{p\log{p}}\big)\cdot \big(\max\limits_i\left\{N_i\right\}\big)^{2k-2+2p}
 \]
  and $\lim\limits_{p\to+\infty}\theta_p=0$.
\end{theorem}
\begin{proof}
    Let $[\phi]\in \mathrm{Out}(G,\mathcal{G})$ be an outer automorphism of finite order. Since $[\phi]$ permutes the cojugacy classes of the $[G_i]$'s, there is some $\kappa=\kappa(k)\in \mathbb{N}$ such that $[\phi]^{\kappa}([G_i])=[G_i]$, i.e. $[\phi]^{\kappa}\in \mathrm{Out}'(G,\mathcal{G})$. By denoting by $g(k)$ the largest order of an element of the symmetric group $\mathrm{Sym}_k$, which is given by Landau's function (see \cite{Landau} for more information), we have that
    \[
    \kappa=\kappa(k)\leq g(k)\leq e^{\frac{k}{e}}.
    \]
    Let $\lambda=\max\big(\mathrm{lcm}\left\{o([\phi]_1),\ldots,o([\phi]_k),o([\phi]_{k+1})\right\}\big)$ where the maximum is taken over all $([\phi]_i)_{1\leq i\leq k+1}\in T(\mathrm{Out}(F_p))\times \prod\limits_{i=1}^k T(\mathrm{Out}(G_i))$. We denote by $Q(p)$ the maximum order of a torsion element in $\mathrm{Out}(F_p)$.  It is known (see \cite{BAO2000437}) that $Q(p)=\mathrm{exp}\big((1+\theta_p)\sqrt{p\log{p}}\big)$, where $\lim\limits_{p\to+\infty}\theta_p=0$. By Lemma \ref{lcm}, it is clear that $\lambda\leq  Q(p)\times\prod_i M_i $ and  $[\phi]^{\kappa\lambda}\in \ker(\pi)$.\par
    Finally, since $[\phi]^{\kappa\lambda}\in \mathrm{Out}_0^{\Delta(T)}(G,\mathcal{G})\cap \mathrm{ker}(\pi)=\mathrm{Out}_0^{\Delta(T)}(G,\mathcal{G})\cap\mathrm{Out}_0(G,\mathcal{G}^{(t)})$ for some $T\in \mathcal{O}(\mathcal{G})$, the embedding (\ref{(5)}) implies that $[\phi]^{\kappa\lambda}\in \prod\limits_{i=1}^k(G_i^{n_i}/Z(G_i))$. Let $\mu$ be the maximal finite order in $\prod\limits_{i=1}^k(G_i^{n_i}/Z(G_i))$. Then, $[\phi]^{\kappa\lambda\mu}=1$, hence $M=\kappa\lambda\mu$ is the desired upper bound. By applying Lemma \ref{lcm} again we get $\mu\leq \prod_i N_i^{n_i}$. Moreover, by Remark \ref{valences}, the value of the sum of the $n_i$'s is at most $2k-2+2p$, hence $\mu\leq  \big(\max\limits_i\left\{N_i\right\}\big)^{2k-2+2p}$.
\end{proof}

\section{Automorphisms of finite extensions of free groups and free products}
\begin{definition}\label{Pprop}
    A group theoretic property $P$ will be called a \textbf{$\mathcal{P}$-property} if: 
    \begin{enumerate}
    \item $P$ is subgroup closed.
    \item If $H$ is a subgroup of finite index in a group $G$ and $H$ has $P$, then $G$ has $P$.
    \item If $G$ has $P$ and $N$ is a finite normal subgroup of $G$, then $G/N$ has $P$.
\end{enumerate}
\end{definition}

Examples of $\mathcal{P}$-properties are: the property of being virtually torsion free, the property of having finite virtual cohomological dimension, the property of being translation
discrete for finitely generated groups, and the property of having subgroups with uniformly bounded orders (Lemma \ref{boundedlemma} below).

\begin{lemma}\label{boundedlemma}
    The property of having finite subgroups with (uniformly) bounded orders is a $\mathcal{P}$-property.
\end{lemma}
\begin{proof}
    It is clear that this property is subgroup closed.\par
    Suppose that the order of every finite subgroup of $H$ is bounded by $M$. Let $G$ be a group such that $H\leq G$ and $|G:H|= m$. We may assume that $G$ has a normal subgroup of finite index, which we still denote by $H$ (and $|G:H|=m$). If $K$ is a finite subgroup of $G$, then we have $\left|\dfrac{HK}{H}\right|\leq |G:H|$ and $|K\cap H|\leq M$. Furthermore, by the second isomorphism theorem, we have
    \[
    \dfrac{K}{K\cap H}\cong \dfrac{HK}{H}
    \]
    hence 
    \[
    \dfrac{|K|}{|H\cap K|}=\left|\dfrac{K}{H\cap K}\right|=\left|\dfrac{HK}{H}\right|\leq |G:H|
    \]
    which gives
    \[
     |K|\leq |H\cap K|\cdot |G:H|=m\cdot M,
    \]
    i.e. every finite subgroup of $G$ has uniformly bounded order.\par
    Let $N$ be a finite normal subgroup of a group $G$ and assume that every finite subgroup of $G$ does not have order more  than $M$. Let $\overline{K}$ be a finite subgroup of $G/N$. There is a subgroup $K\leq G$ such that $\overline{K}=\frac{K}{N}\leq \frac{G}{N}$. The groups $N$ and $\overline{K}$ are finite, hence $K$ is also finite with $|K|\leq M$. Therefore, we have
    \[
    |\overline{K}|=\dfrac{|K|}{|N|}<|K|\leq M.
    \]
\end{proof}

The following lemma has been proved useful for the study of $\mathrm{Out}(G)$ in the case where the group $G$ has a finite index subgroup $N$ with trivial center (see for example \cite{SykPap}).

\begin{lemma}[{{\cite[Lemma 3.3.]{SykPap}}}]\label{Plemma}
    Let $G$ be a finitely generated group, $N$ a normal subgroup of $G$ of finite
index with trivial center and $\mathcal{P}$ a group theoretic property as above. If $\mathrm{Out}(N)$ satisfies $\mathcal{P}$, then so does $\mathrm{Out}(G)$. 
\end{lemma}
It is well known that finite subgroups of $\mathrm{Out}(F_n)$ have uniformly bounded orders. More precisely:
\begin{theorem}[\cite{WZ}]\label{Out(Fn)}
    The
maximal order of a finite subgroup of $\mathrm{Out}(F_n)$ is $2^n\cdot n!$.
\end{theorem}

Combining Lemma \ref{Plemma} with Theorem \ref{Out(Fn)} we obtain the following.
\begin{corollary}\label{fbf}
    Let $G$ be a finitely generated free-by-finite group $G$. Then every finite subgroup of $\mathrm{Out}(G)$ has uniformly bounded order. In particular, if $G$ is not virtually cyclic, then the order of a finite subgroup of $\mathrm{Out}(G)$ is at most $2^{\mathrm{rank}(F)}\cdot \mathrm{rank}(F)!\cdot |G:F|$, where $F$ is a finite index, normal, free subgroup of $G$.
\end{corollary}
\begin{proof}
    $G$ contains a finitely generated free (normal) subgroup $F$ of finite index $|G:F|$. If $F$ has rank $1$, then $\mathrm{Out}(G)$ is a finite group. If $F$ has rank more than $1$, then 
its center $Z(F)$ is trivial and hence Lemmas \ref{boundedlemma}, \ref{Plemma} with Theorem \ref{Out(Fn)} imply that $\mathrm{Out}(G)$ has finite subgroups of uniformly bounded orders. In fact, from the proof of Lemma \ref{boundedlemma}, we see that every finite subgroup of $\mathrm{Out}(G)$ has order no more than $2^{\mathrm{rank}(F)}\cdot \mathrm{rank}(F)!\cdot |G:F|$.
\end{proof}

We also extend Theorem \ref{main} to finite extensions of free products.
\begin{corollary}\label{vfp}
    Let $G=G_1\ast\ldots G_k\ast F_p$ be a free product such that the $G_i$'s and $\mathrm{Out}(G_i)$'s have uniformly bounded orders of finite subgroups. If $\widetilde{G}$ contains $G$ as a normal subgroup of finite index, then $\mathrm{Out}(\widetilde{G})$ has uniformly bounded orders of finite subgroups. 
\end{corollary}
\begin{proof}
    Since $Z(G)=1$, appealing to Lemmas \ref{boundedlemma} and \ref{Plemma} complete the proof.
\end{proof}
\printbibliography

\begin{multicols}{2}

\noindent
Ioannis Papavasileiou\\
National and Kapodistrian University of Athens\\
13giannispapav@hotmail.gr\\
https://orcid.org/0009-0009-5421-3013

\noindent
Dionysios Syrigos\\
Winterhur way, rg21 7uq, Basingstoke, uk\\
Dionisissyrigos@gmail.com\\
https://orcid.org/0000-0002-7876-2641
\end{multicols}
\end{document}